\documentclass[12pt,a4paper]{extarticle}
\usepackage{float}
\usepackage{amsmath, amsthm, mathtools, amssymb}
\usepackage{graphicx}
\numberwithin{equation}{section} % generates equation numbers

\DeclarePairedDelimiter{\ceil}{\lceil}{\rceil}

\usepackage{cleveref}

\newtheorem{thm}{Theorem}[section] % renumbers in each section

\newtheorem{obs}[thm]{Observation}

\theoremstyle{definition}
\newtheorem{defn}[thm]{Definition}

\newtheorem{rmk}[thm]{Remark}

\crefname{theorem}{theorem}{theorems}
\Crefname{theorem}{Theorem}{Theorems}

\newcommand{\Wlg}{Without loss of generality}
\newcommand{\wlg}{without loss of generality}

\newcommand{\Everify}{It can be easily verified that}

\usepackage{authblk} %allows inserting author affiliation

\author  {Brian D'Souza \thanks  {brian.dsouza@gcsanquelim.ac.in} }

\author {Jessica Pereira \thanks{zaydegrace@unigoa.ac.in}}

\affil {\small School of Physical and Applied Sciences, Goa University,\\ Taleigao Plateau, Goa 403206, India. }

\title{\vspace{-4 cm} Some Results on Additively Graceful Signed Stars and Double Stars \thanks{This paper is a preprint and is available on arXiv.}}
\date{\vspace{-2 cm}}

\begin{document}

	\maketitle

	\begin{abstract}
		We study additively graceful labelings of signed graphs on stars and double stars. While the case of signed stars is straightforward, the problem becomes significantly more intricate for signed double stars. We obtain a characterization of additively graceful signed stars, while, for several sub-classes of additively graceful signed double stars, we establish existence, uniqueness and non-existence results.
	\end{abstract}
	{\it AMS Subject of classification}: \textbf{05C78, 05C22}


\section{Introduction}

Throughout this paper, a graph $G$ refers to a finite, nonempty set of objects called \emph{vertices}, together with a collection of unordered pairs of distinct vertices called \emph{edges}. The sets of vertices and edges are denoted by $V(G)$ and $E(G)$, respectively, and their cardinalities by $p$ and $q$. In this case, $G$ is called a $(p,q)$ graph. Standard terminology from graph theory used in this paper can be found in \cite{chartrand2010graphs}.

An injective map from $V(G)$ to $\mathbb{Z}$ is called a {\it vertex labeling} of $G$ while an {\it edge labeling} is an injective map from $E(G)$ to $\mathbb{Z}$. Although vertex and edge labelings may be defined independently, we will often be interested in deriving edge labels from a given vertex labeling via a specified rule. The resulting edge labeling is called an {\it induced edge labeling}, and the corresponding edge labels are called {\it induced edge labels}. A comprehensive survey of graph labelings can be found in Gallian \cite{DynamicSurveyGallian}.
	\begin{defn}
	\cite{rosa1967certain,
		golomb1972number}
	Let $G$ be a $(p,q)$ graph. A {\it graceful labeling} of $G$ is an injection $f:V(G)\rightarrow \{0,1,\dots, q\}$ such that when each edge $uv$ is assigned the label $f^*(uv) = |f(u)-f(v)|$, the resulting induced edge labels are all distinct. A graph which admits such a labeling is called a {\it graceful graph}.
	\end{defn}
	
	Rosa \cite{rosa1967certain} referred to such a labeling as a $\beta$-valuation, and Golomb \cite{golomb1972number} subsequently termed it a graceful labeling. The following is a well-known result from the literature.
	
	\begin{thm}
	\label{thm:caterpillar is graceful}
	Every caterpillar is graceful.
	\end{thm}

In this paper, the investigation focuses on two classes of caterpillars, namely stars and double stars. The complete bipartite graph \(K_{1,t}\) is called a \emph{star}, whereas, given two stars \(K_{1,r}\) and \(K_{1,s}\), each with a chosen central vertex, a \emph{double star} is defined as the graph obtained by joining these two central vertices by an edge. Vertices of degree 1 shall be referred to as \emph{pendant vertices} or \emph{pendants} while their lone incident edge shall be called a \emph{pendant edge}.
\\\\
Hegde \cite{hegde1989additively} introduced the concept of additively graceful  graphs, which is a variation of harmonious labelings.
\begin{defn}
	\cite{hegde1989additively}
	Let $G$ be a $(p,q)$ graph. An {\it additively graceful labeling} of $G$ is an injection $f:V(G)\rightarrow \{0,1,\dots, \ceil{\frac{q+1}{2}}\}$, such that when each edge $uv$ is assigned the label $f^*(uv) = f(u)+f(v)$, we have $f^*(E(G))=\{1,2,\dots,q\}$. A graph which admits such a labeling is called an {\it additively graceful graph}.
\end{defn}

Building on the work in \cite{hegde1989additively}, Pereira et al. \cite{pereira2023additively} generalized the concept of additively graceful graphs by defining additively graceful signed graphs. A {\it signed graph} $S$, is a $(p,q)$ graph in which certain $m$ edges are specified as positive and the remaining $n$ edges are specified as negative. We call $S$ a $(p,m,n)$ signed graph and denote the set of positive and negative edges by $E^+(S)$ and $E^-(S)$, respectively. $S$ shall be called {\it all positive} (or {\it all negative}) if all its edges are positive (or negative). In diagrams of signed graphs, we will use a {solid line} to represent a positive edge, while a {dashed line} shall represent a negative edge (see Figure \ref{fig:fig12-used-in-second-paper}).

\begin{defn} \cite{pereira2023additively}
	Let $S=(V,E)$ be a $(p,m,n)$ signed graph.
	Let $f:V \rightarrow \{0,1,\dots, m+\lceil\frac{(n+1)}{2}\rceil\}$ be an injective mapping and let the induced edge function be defined as $f^*(uv)=f(u)+f(v), ~ \forall ~ uv \in E^-(S)$ and $f^*(uv)=|f(u)-f(v)|,~\forall~ uv \in E^+(S)$. If $\{f^*(uv):uv \in E^-(S)\}=\{1, 2, ..., n\}$ and $\{f^*(uv):uv \in E^+(S)\}=\{1, 2, ..., m\}$, then $f$ is called an {\it additively graceful labeling} of $S$. A signed graph which admits such a labeling is called an {\it additively graceful signed graph}.
\end{defn}

The following complete characterizations for additively graceful signed paths and signed cycles have been obtained in \cite{2025PathsCycles}.

\begin{thm} \emph{\cite{2025PathsCycles}} \label{Characterization theorem for signed paths}
	A signed path $S$ with $n$ negative edges is additively graceful if and only if $n \leq 2$ and $S$ has at most one negative section, except that the four signed paths in \Cref{fig:fig12-used-in-second-paper} are not additively graceful.
\end{thm}

\begin{figure}
	\centering
	\includegraphics[width=0.7\linewidth]{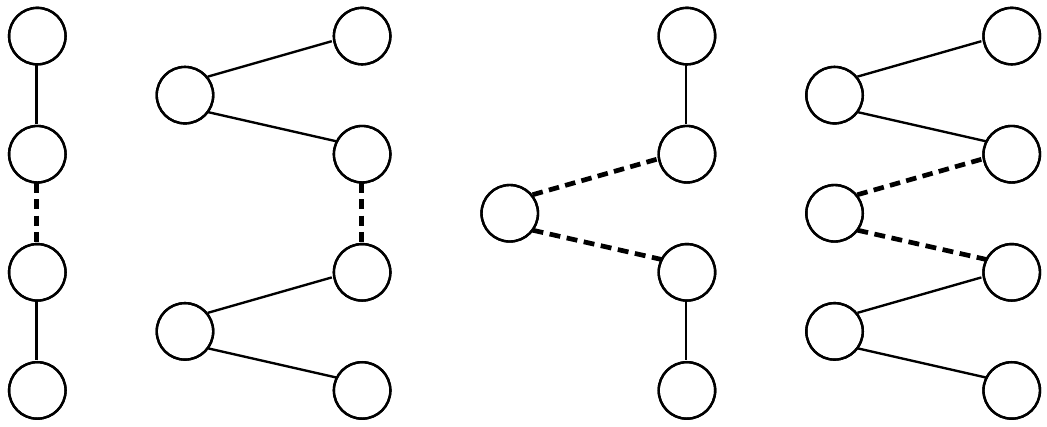}
	\caption{The only signed paths, with a single negative section, that are not additively graceful for $n \leq 2$.}
	\label{fig:fig12-used-in-second-paper}
\end{figure}

\begin{thm}
	\emph{\cite{2025PathsCycles}} 
	\label{characterization of add grac signed cycles}
	A signed cycle $S$ with $m$ positive and $n$ negative edges, is additively graceful if and only if one among the following 4 conditions is satisfied,
	\begin{enumerate}
		\item $n=0$ and $ m\equiv 0$ or $3 \pmod 4$.
		\item $n=1$ and $ m\equiv 1$ or $2 \pmod 4$.
		\item $n=2$, $ m\equiv 1$ or $2 \pmod 4$ and $S$ contains a single negative section.
		\item $S$ is the all negative signed cycle on $C_3$.
	\end{enumerate}		
\end{thm}

We undertake, in this paper, a systematic investigation of additively graceful signed stars and double stars. By a {\it signed star} or a {\it signed double star} we mean a signed graph whose underlying graph (graph obtained by ignoring the parity of edges), is a star or a double star respectively. Our primary objective is to identify conditions that guarantee the existence, uniqueness or non-existence of additively graceful labelings on signed stars and double stars.

\begin{obs}
	\label{obs:g=+-f+-k gives f*=g*}
	\emph{\cite{2025PathsCycles}}
	Let $f$ be a vertex labeling in graph $G$ which induces an edge labeling given by $f^*(uv)= |f(u)-f(v)|$. Then for any integer $k$, both the vertex labelings $g= f +k$ as well as $g=-f+k$ satisfy $f^* = g^*$.
\end{obs}

Let $f$ be a vertex labeling of a $(p,q)$ graph $G$ and let $M$ be an upper bound for $f$ on $V(G)$. Then $M-f$ is called the \emph{complementary labeling of $f$ with respect to $M$}. In particular, if $M= \max_{v \in V(G)}{f(v)}$, then $M-f$ is simply called the \emph{complementary labeling} of $f$.

\begin{obs}
	Let $f$ be a vertex labeling in graph $G$ which induces an edge labeling given by $f^*(uv)= |f(u)-f(v)|$. If $M > f(v)~ \forall~ v \in V(G)$ and if $g$ is the complementary labeling of $f$ with respect to $M$, then, $f^*=g^*$.
\end{obs}

For a signed graph $S$, with a vertex labeling $f$, we shall use the notation  $f( w_1,w_2,\dots,w_k )$ to denote the ordered tuple $( f(w_{1}), f(w_{2}), \dots, f(w_{k}) )$  of vertex labels where $ w_1,w_2,\dots,w_k $ are some $k$ vertices in $S$. Similarly, if $f^*$ is the induced edge labeling then, $f^*(u_1v_1,u_2v_2,\dots,u_kv_k )$ shall denote $(f^*(u_1v_1),f^*(u_2v_2),\dots,f^*(u_kv_k) )$ where $ u_1v_1,u_2v_2,\dots,u_kv_k $ are some $k$ edges in $S$.

\section{Basic Results}
\label{sec: basic results}
In this section we present some preliminary observations as well as a characterization of additively graceful signed stars.
The following are a couple of very basic observations, but we state them here since they will be used repeatedly in our discussions. \Cref{grace implies add_grace_signed graph} follows directly from the definition of graceful graph and additively graceful signed graph. 

\begin{obs} \label{grace implies add_grace_signed graph}
	If a graph $G$ is graceful then as a signed graph with $n=0$ it is an additively graceful signed graph.
\end{obs}

 Also, since $0+1$ and $0+2$ are the only ways in which 1 and 2 can be expressed as the sum of two distinct, non-negative integers, we obtain \Cref{obs: two negative edges must be adjacent in add grace sigraph}.

\begin{obs} \label{obs: two negative edges must be adjacent in add grace sigraph}
	In an additively graceful labeling of a signed graph $S$ having two or more negative edges, the negative edges labeled $1$ and $2$ must be adjacent to the vertex labeled $0$.
\end{obs} 

\begin{thm}
	\label{thm: add grace tree requires n<=2}
	If a $(p,m,n)$ signed graph $S$ on a tree is additively graceful then $n\leq2$. Further, if $n=1,2$ and if $f$ is an additively graceful labeling  of $S$ then, $f(V(S))=\{0,1,\dots,m+\ceil{\frac{n+1}{2}}\}$
\end{thm}

\begin{proof}
	Suppose $n>2$ and suppose $f:V(S) \rightarrow \{0,1,\dots,m+\ceil{\frac{n+1}{2}} \}$ is a additively graceful labeling of $S$. Since $n>2$ hence, $\ceil{\frac{n+1}{2}} \leq n-1$. This implies that there are at most $m+n$ vertex labels to choose from, so as to label the $m+n+1$ vertices, which is impossible.
	
	Further, if $n=1$ or $2$ then, $m+\ceil{\frac{n+1}{2}}=m+n$. Hence, there are exactly $m+n+1$ vertex labels available to $f$, so as to label as many vertices. The theorem follows.
\end{proof}
		
		We say that two vertex labelings $f$ and $g$, on a graph (or signed graph) are {\it  equivalent up to complementary labeling} if one is the complementary labeling of the other with respect to some $M$. In this case we write $f \sim g$.
		
		\begin{figure}[h]
			\centering
			\includegraphics[width=1\linewidth]{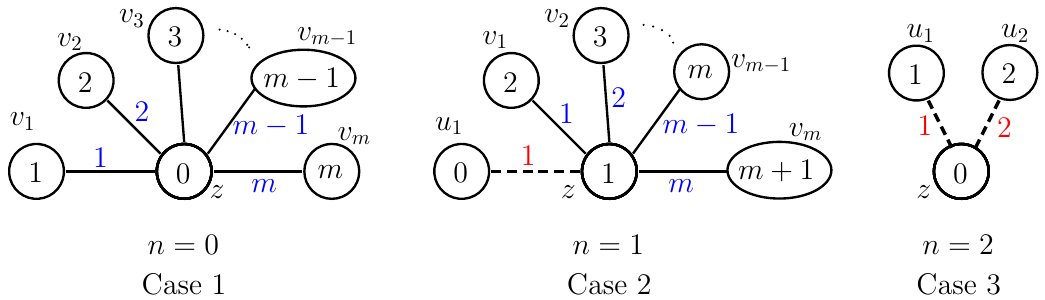}
			\caption{Additively graceful signed stars}
			\label{fig:fig2018used}
		\end{figure}

		\begin{thm}
		\label{thm: stars are additively graceful iff n=1}
		A $(p,m,n)$ signed graph $S$ on a star $K_{1,k}$ is additively graceful if and only if $n=0,1$ or $S$ is the path $P_3$, both of whose edges are negative.			
		Further, whenever  $n=1$ or $2$, the additively graceful labelings obtained corresponding to each permissible value of $m$, are unique up to rearrangement of labels among the $u_i$'s and $v_i$'s, whereas, when $n=0$, there are exactly 2 distinct labelings, for each value of $m \geq 2$,
		up to equivalence under complementary labelings and rearrangement of labels among the $u_i$'s and $v_i$'s. 
		\end{thm}
		
		\begin{proof}
			Let $u_iz$, $i=1,2,\dots,n$ be the negative edges and let  $v_iz$, $i=1,2,\dots,m$ be the positive edges in $S$.
			\\\\
			Suppose $n=0,1$ or $S$ is the all negative signed path on $P_3$.
			\\\\
			{\bf Case 1}: $n=0$.
			\\
			Define $f:V(S)\rightarrow\{0,1,2,\dots,m+1\}$ by $f(z)=0$ and $f(v_i)=i$, for $i=1,2,\dots, m$. \Everify, $f$ is an additively graceful labeling.
			\\\\
			{\bf Case 2}: $n=1$.
			\\
			Define $f:V(S)\rightarrow\{0,1,2,\dots,m+1\}$ by $f(z)=1$, $f(u_1)=0$ and $f(v_i)=i+1$, for $i=1,2,\dots, m$. \Everify, $f$ is an additively graceful labeling.
			\\\\
			{\bf Case 3}: $S$ is the all negative signed path on $P_3$.
			\\
			Let $P_3=u_1,z,u_2$. Define $f:V(S)\rightarrow\{0,1,2\}$ by $f(u_1,z,u_2)=(1,0,2)$. \Everify, $f$ is an additively graceful labeling.
			\\\\
			Conversely, let $f:V(S)\rightarrow\{0,1,2,\dots,m+\ceil{\frac{n+1}{2}} \}$ be a additively graceful labeling of $S$. By \Cref{thm: add grace tree requires n<=2}, we know that $n\leq2$. If we assume that $n=2$, then by \Cref{obs: two negative edges must be adjacent in add grace sigraph}, $f(z)=0$ and \wlg, $f(u_1,u_2)=(1,2)$. Notice that, $S$ cannot contain a positive edge because in such a case, it would be impossible to obtain the edge label 1 on any positive edge in $S$. It follows that, $S$ is the all negative signed path on $P_3$.
			\\\\
			Further, consider $f$ as defined in Case 1 above. Since $n=0$ hence, $m+\ceil{\frac{n+1}{2}}=m+1$. 
			Define $g:V(S)\rightarrow\{0,1,2,\dots,m+1\}$ by $g(z)=1$, $g(v_1)=0$ and $g(v_i)=i+1$, for $i=2,\dots, m$. \Everify, $g$ is an additively graceful labeling of $S$. The maximum distance between vertex labels in $f$ is $m$ while the same is $m+1$ in $g$. Since complementary labelings preserve this maximum gap, we conclude that $f \nsim g$. Now let $\tilde{f}$ be the complementary labeling of $f$ with respect to $m$. \Everify, $\tilde{f}$ and the complementary labelings of  $f$, $\tilde{f}$ and $g$, with respect to $m+1$ are also additively graceful labelings of $S$.
			We leave it to the reader to verify that the additively graceful labelings obtained here and in the 3 cases above, are the only ones possible, up to rearrangement of labels among the $u_i$'s and $v_i$'s. 
		\end{proof}

		\section{Additively Graceful Signed Graphs on Double Stars with $n \leq 1$}
		\label{sec: n <=1}
		In this section we investigate additively graceful signed double stars with less than 2 negative edges. 
		The following theorem deals with the case $n=0$ and is an immediate consequence of \Cref{grace implies add_grace_signed graph}  and \Cref{thm:caterpillar is graceful}
		
		\begin{thm}
			\label{thm: double star with n=0 is add grace}
			Every $(p,m,0)$ signed graph $S$ on a double star is additively graceful.
		\end{thm}
		
		\noindent We now investigate $(p,m,1)$ additvely graceful signed double stars.
		
		\begin{thm}
			\label{thm: double star with one pendant neg edge is add grace}
			Every $(p,m,1)$ signed graph $S$ on a double star, with a single pendant negative edge is additively graceful.
		\end{thm}
		
		\begin{proof}
			\Wlg, we may assume that $S$ consists of a positive edge $z_1z_2$, to which are attached positive edges $z_1v_i$, $i=1,2, \dots, l$ and $z_2w_i$, $i=1,2, \dots, r$, while $z_1u$ is the lone negative edge.

			Define $f:V(S)\rightarrow\{0,1,\dots,l+r+2\}$ as follows. 
			\begin{equation*}
				f(x)= \begin{cases}
					0 & \text{for } x=u \\
					1 & \text{for } x=z_1\\
					r+2+i & \text{for } x=v_i~~~i=1,2,\dots,l\\
					r+2 & \text{for } x=z_2 \\
					1+i & \text{for } x=w_i~~~i=1,2,\dots,r\\
				\end{cases}
			\end{equation*}
			\Everify, $f(V(S))=\{0,1,\dots,l+r+2\}$ and $f$ is an additively graceful labeling of $S$.
		\end{proof}

		\begin{thm}
			\label{thm: double star with exactly one neg edge in middle is add grace}
			Let $S$ be the $(p,m,1)$ signed graph consisting of a negative edge, to which are attached $l$ positive pendant edges at one end and $r$ positive pendant edges at the other. Then, $S$ is additively graceful if and only if either $l=0$ or $r=0$.
		\end{thm}
		
		\begin{proof}
			Let $z_1z_2$ be the negative edge, to which are attached positive edges $z_1v_i$, $i=1,2, \dots, l$ and $z_2w_i$, $i=1,2, \dots, r$.
			Suppose $l\geq1$ and $r\geq1$ and let $f:V(S) \rightarrow \{0,1,\dots,l+r+1\}$ be an additively graceful labeling of $S$. \Wlg, we may assume that $f(z_1)=0$ and $f(z_2)=1$. The only way to then obtain the edge label 1 on a positive edge is to  set $f(w_i)=2$ for some $i$. It follows that, the only way to now obtain the edge label 2 on a positive edge is to set $f(w_j)=3$ for some $j$, and so on. Hence,\wlg, we may set $f(w_i)=i+1$ for $i=1,2,\dots,r$. It is now impossible to obtain the required edge label $r+1$ on a positive edge.
			
			Conversely, if $l=0$ or $r=0$ then, $S$ reduces to a star with one negative edge, which by \Cref{thm: stars are additively graceful iff n=1}, is additively graceful.
		\end{proof}
		
		\section{ Additively Graceful Signed Double Stars with 2 Negative Edges}
		\label{sec: n=2}
		In this section, we investigate additively graceful signed double stars with two negative edges. We obtain the following existence and non-existence results.
		
		\begin{thm}
			\label{thm: double star with 2 neg pendants and r<2 is not grace elegant}
			Let $S$ be the $(p,m,2)$ signed graph consisting of a positive edge, to which are attached $l$ positive pendant edges and 2 negative pendant edges at one end as well as $r$ positive pendant edges at the other end.
			If $r<2$ then, $S$ is not additively graceful.
		\end{thm}
		
		\begin{proof}
			Let $z_1z_2$ be the positive edge, to which are attached positive edges $z_1v_i$, $i=1,2, \dots, l$ and $z_2w_i$, $i=1,2, \dots, r$, as well as negative edges $z_1u_1$ and $z_1u_2$.
			Suppose $f$ is an additively graceful labeling of $S$. By \Cref{obs: two negative edges must be adjacent in add grace sigraph} and \wlg, $f(z_1,u_1,u_2)=(0,1,2)$. It follows that, $f(z_2)\geq 3$. This implies that $f^*(z_1z_2) \geq 3$, which in turn requires the existence of positive edge labels 1 and 2. These can only be generated on the remaining edges adjacent to $z_2$. It follows that, $r\geq 2$.
		\end{proof}
		
		\begin{thm}
			\label{thm: double star with 2 neg pendants and l=r-2 is grace elegant}
			Let $S$ be the $(p,m,2)$ signed graph consisting of a positive edge, to which are attached $l$ positive pendant edges and 2 negative pendant edges at one end as well as $r$ positive pendant edges at the other end.	
			If $l=r-2$ then, $S$ is additively graceful.
		\end{thm}
		
		\begin{proof}
			Let $z_1z_2$ be the positive edge, to which are attached positive edges $z_1v_i$, $i=1,2, \dots, l$ and $z_2w_i$, $i=1,2, \dots, r$, as well as negative edges $z_1u_1$ and $z_1u_2$.
			We assume that,  $l=r-2$ and define $f:V(S)\rightarrow\{0,1,\dots,2r+1\}$ as follows. $f(z_1,u_1,u_2)=(0,1,2)$ 
			\begin{equation*}
				\label{eqn: double star pos edge in middle and l+n=r}
				f(x)= \begin{cases}
					
					r+1 & \text{for } x=z_2 \\
					r+1+i & \text{for } x=v_i~~~i=1,2,\dots,l\\
					i+2 & \text{for } x=w_i~~~i=1,2,\dots,l\\
					r+1+i & \text{for } x=w_i~~~i=l+1,l+2
				\end{cases}
			\end{equation*}
			\Everify, $f(V(S))=\{0,1,\dots,2r+1\}$ and $f$ is an additively graceful labeling of $S$.
		\end{proof}
		
		\begin{thm}
			\label{thm: double star with 2 neg pendants and l=odd l>=3 r=3 is not add graceful}
			Let $S$ be the $(p,m,2)$ signed graph consisting of a positive edge, to which are attached $l$ positive pendant edges and $2$ negative pendant edges at one end as well as $3$ positive pendant edges at the other end.
			If $l$ is odd and $l\geq3$ then, $S$ is not additively graceful.
		\end{thm}

		\begin{proof}
			Let $z_1z_2$ be the positive edge, to which are attached positive edges $z_1v_i$, $i=1,2, \dots, l$ and $z_2w_i$, $i=1,2,3$, as well as negative edges $z_1u_1$ and $z_1u_2$.
			Suppose $f:V(S)\rightarrow\{0,1,\dots,l+6\}$ is an additively graceful labeling of $S$. By \Cref{obs: two negative edges must be adjacent in add grace sigraph} and \wlg, $f(z_1,u_1,u_2)=(0,1,2)$. Now $f$ needs to assign the labels $3,4,\dots,l+6$ to the vertices $w_1,w_2,w_3,z_2,v_1,v_2\dots, v_l$.
			\Wlg, we may assume that, $f(w_2)=l+5$ and $f(w_3)=l+6$, since $m=l+4$ and hence these labels cannot be assigned to vertices adjacent to $z_1$.
			\\\\
			Now if $f(z_2)<l+3$ then neither edge label 1 nor 2 can occur on edges $z_2w_2$ and $z_2w_3$. It follows that atleast one among edge labels 1 or 2 must occur on a positive edge adjacent to $z_1$, which is impossible. Further, if $f(z_2)=l+3$ then, $f(w_1)=l+2 \text{ or } l+4$, as this is the only way to obtain edge label 1 on a positive edge. However, 
			since $l\geq 3$ hence, $l+2$ and $l+4$, both of which need to occur as edge labels on a positive edge, are not equal to $f^*(z_2w_2)$ or $f^*(z_2w_3)$. It follows that one among $l+2$ or $l+4$ cannot be obtained as an edge labels on a positive edge. Therefore $f(z_2)=l+4$, which is odd.
			
			We now need to assign the labels $3,4,\dots,l+3$ to the vertices $w_1,v_1,v_2\dots, v_l$ so as to again obtain edge labels $3,4,\dots,l+3$, on the remaining positive edges. But because $f(z_2)$ is odd hence, for every $x \in \{3,4,\dots,l+3\}$, $x\neq f(z_2)-x$. It follows that if $f(w_1)=x$ then it is impossible to obtain $x$ as an edge label on a positive edge. Hence $f$ cannot be an additively graceful labeling.
		\end{proof}
		
		\begin{thm}
			\label{thm: double star with 2 neg pendants and r=2 is add grace}
			Let $S$ be the $(p,m,2)$ signed graph consisting of a positive edge, to which are attached $l$ positive pendant edges and 2 negative pendant edges at one end as well as $2$ positive pendant edges at the other end.
			Then, $S$ is additively graceful. Further for each $l \in \mathbb{N}\cup\{0\}$, the additively graceful labeling obtained is unique up to rearrangement of labels among the $u_i$'s, $v_i$'s and $w_i$'s.
		\end{thm}
		
		\begin{proof}
			Let $z_1z_2$ be the positive edge, to which are attached positive edges  $z_2w_1$, $z_2w_2$ and $z_1v_i$, $i=1,2, \dots, l$, as well as negative edges $z_1u_1$ and $z_1u_2$.
			Define $f:V(S)\rightarrow\{0,1,\dots,l+5\}$ as follows. $f(z_1,u_1,u_2)=(0,1,2)$
			\begin{equation*}
				\label{eqn: double star pos edge in middle and r=n}
				f(x)= \begin{cases}
					
					i+2 & \text{for } x=v_i~~~i=1,2,\dots,l\\
					l+3 & \text{for } x=z_2 \\
					l+4 & \text{for } x=w_1 \\
					l+5 & \text{for } x=w_2
				\end{cases}
			\end{equation*}
			\Everify, $f(V(S))=\{0,1,\dots,l+5\}$ and $f$ is an additively graceful labeling of $S$.
			\\\\
			Further, let $g$ be any additively graceful labeling of $S$. \Cref{obs: two negative edges must be adjacent in add grace sigraph}, implies that $g(z_1)=0$ and $g$ assigns labels $1$ and $2$ to vertices $u_1$ and $u_2$.
			Now, $g$ is forced to assign labels $l+4$ and $l+5$ to vertices $w_2$ and $w_3$, since, $m=l+3$ and hence these labels cannot be assigned to vertices adjacent to $z_1$.
			Consequently, the only way to obtain edge labels 1 and 2 on positive edges is to set $g(z_2)=l+3$. The remaining labels $3,4,\dots,l+2$ can only be assigned to vertices $v_i$, $i=1,2,\dots,l$. Hence $g=f$, up to rearrangement of labels among the $u_i$'s, $v_i$'s and $w_i$'s.
		\end{proof}
		
			\begin{thm}
			\label{thm: double star with 2 neg pendants and l=0 r>2 is not grace elegant} Let $S$ be the $(p,m,2)$ signed graph consisting of a positive edge, to which are attached 2 negative pendant edges at one end and $r$ positive pendant edges at the other end.
			If $r>2$ then, $S$ is not additively graceful.
		\end{thm}

		\begin{proof}
			Let $z_1z_2$ be the positive edge, to which are attached positive edges  $z_2w_i$, $i=1,2, \dots, r$, as well as negative edges $z_1u_1$ and $z_1u_2$.
			Suppose $f:V(S)\rightarrow\{0,1,\dots,r+3\}$ is an additively graceful labeling of $S$. By \Cref{obs: two negative edges must be adjacent in add grace sigraph} and \wlg, $f(z_1,u_1,u_2)=(0,1,2)$. Now $f$ needs to assign the labels $3,4,\dots,r+3$ to the vertices $z_2,w_1,w_2,\dots, w_r$. Since $f^*(z_1z_2)$ cannot exceed $r+1$ hence,
			\begin{equation*}
				f(z_2)\leq r+1 
			\end{equation*}
			Since $r>2$ hence, $r+3\geq 6$, which forces $f(w_i)=6$ for some $i$. This implies that $f(z_2)\neq 3$, since otherwise we obtain the contradiction $f^*(z_1z_2)=3=f^*(z_2w_i)$. Therefore,
			\[
			3<f(z_2)<r+3
			\]
			It follows that, $f(z_2)-1$ and $f(z_2)+1$ are assigned to some vertices $w_i$ and $w_j$. This in turn leads to the  contradiction $f^*(z_2w_i)=f^*(z_2w_j)=1$.
		\end{proof}
		
			\begin{thm}
			\label{thm: double star with 2 neg pendants and l=2<=r is grace elegant}
			Let $S$ be the $(p,m,2)$ signed graph consisting of a positive edge, to which are attached $2$ positive pendant edges and 2 negative pendant edges at one end as well as $r$ positive pendant edges at the other end.
			Then, for every $r\geq2$, $S$ is additively graceful.
		\end{thm}
		
		\begin{proof}
			Let $z_1z_2$ be the positive edge, to which are attached positive edges  $z_1v_1$, $z_1v_2$ and $z_2w_i$, $i=1,2, \dots, r$, as well as negative edges $z_1u_1$ and $z_1u_2$.
			\\\\
			Define $f:V(S)\rightarrow\{0,1,\dots,r+5\}$ as follows. $f(z_1,u_1,u_2)=(0,1,2)$
			\begin{equation*}
				\label{eqn: double star pos edge in middle and r=n}
				f(x)= \begin{cases}
					
					r+1 & \text{for } x=v_{1} \\
					r+2 & \text{for } x=v_{2}\\	
					r+3 & \text{for } x=z_2 \\
					i+2 & \text{for } x=w_i~~~i=1,2,\dots,r-2\\
					r+4 & \text{for } x=w_{r-1} \\
					r+5 & \text{for } x=w_{r}
				\end{cases}
			\end{equation*}
			\Everify, $f(V(S))=\{0,1,\dots,r+5\}$ and $f$ is an additively graceful labeling of $S$.
		\end{proof}
		
		\begin{thm}
			\label{thm: double star with 2 neg pendants with r=even and r l>=2 is grace elegant}
			Let $S$ be the $(p,m,2)$ signed graph consisting of a positive edge, to which are attached $l$ positive pendant edges and 2 negative pendant edges at one end as well as $r$ positive pendant edges at the other end.
			If $l,r\geq2$ and $r$ is even then, $S$ is additively graceful.
		\end{thm}
		
		\begin{proof}
			Let $z_1z_2$ be the positive edge, to which are attached positive edges $z_1v_i$, $i=1,2, \dots, l$ and $z_2w_i$, $i=1,2, \dots, r$, as well as negative edges $z_1u_1$ and $z_1u_2$.
			Since \Cref{thm: double star with 2 neg pendants and r=2 is add grace} proves the result for $r=2$ hence, we may assume $r\geq4$. Also since $r$ is even hence, $r=2t$ for some $t \in \mathbb{N}$.
			\\\\			
			Define $f:V(S)\rightarrow\{0,1,\dots,2t+l+3\}$ as follows. $f(z_1,u_1,u_2)=(0,1,2)$
			
			\begin{equation*}
				\label{eqn: double star with 2 neg pendants with r=even and r,l>=2 is grace elegant}
				f(x)= \begin{cases}
					
					t+1+i & \text{for } x=v_i~~~i=1,2,\dots,l-2 \\
					2t+i & \text{for } x=v_i~~~i=l-1,~l\\	
					2t+l+1 & \text{for } x=z_2 \\
					2+i & \text{for } x=w_i~~~i=1,2,\dots,t-1\\
					l+i & \text{for } x=w_i~~~i=t,t+1,\dots,2t-2 \\
					l+3+i & \text{for } x=w_i~~~i=2t-1,~2t
				\end{cases}
			\end{equation*}
			\Everify, $f(V(S))=\{0,1,\dots,2t+l+3\}$ and $f$ is an additively graceful labeling of $S$.
		\end{proof}
		
			\begin{thm}
			\label{thm: double star with 2 neg pendants and l=1 r>4 is not grace elegant}
			Let $S$ be the $(p,m,2)$ signed graph consisting of a positive edge, to which are attached $2$ negative pendant edges and 1 positive pendant edge at one end as well as $r$ positive pendant edges at the other end.
			If $r\geq5$ then, $S$ is not additively graceful.
		\end{thm}
		
		\begin{proof}
			Let $z_1z_2$ be the positive edge, to which are attached positive edges $z_1v_1$ and  $z_2w_i$, $i=1,2, \dots, r$, as well as negative edges $z_1u_1$ and $z_1u_2$.
			Suppose $f:V(S)\rightarrow\{0,1,\dots,r+4\}$ is an additively graceful labeling of $S$. By \Cref{obs: two negative edges must be adjacent in add grace sigraph} and \wlg, $f(z_1,u_1,u_2)=(0,1,2)$. Now, $f$ needs to assign the labels $3,4,\dots,r+4$ to the vertices $v_1,z_2,w_1,w_2,\dots, w_r$.
			Since $r\geq 5$ hence, $r+4\geq9$.
			\\\\
			If $f(z_2)=3$ then to avoid duplication of edge label 3 we must set $f(v_1)=6$. This in turn ensures that $f(w_i)=9$ for some $i$, causing the duplication of edge label 6. Therefore, $f(z_2)\neq3$
			\\\\
			If $f(z_2)=4$ then to avoid duplication of edge label 4 we must set $f(v_1)=8$. This in turn ensures that $f(w_i,w_j)=(3,5)$ for some $i,j$, causing the duplication of edge label 1. Hence $f(z_2)\neq4$
			\\\\
			Further, since $f^*(z_1z_2)=f(z_2)$ and $m=r+2$ hence, $f(z_2)\leq r+2$. Therefore, we have,
			\[
			5 \leq f(z_2) \leq r+2
			\]
			It follows that, for any value of $f(z)$, the vertex labels assigned to vertices $v_1,w_1,w_2,\dots, w_r$ includes  $f(z)-1$, $f(z)-2$, $f(z)+1$ and $f(z)+2$, only one among which can be $f(v_1)$. This implies that either 1 or 2 will occur twice as labels on edges $z_2w_i$ and $z_2w_j$ for some $i,j$. This gives a contradiction.
		\end{proof}
		
		\begin{figure}[h]
			\centering
			\includegraphics[width=1\linewidth]{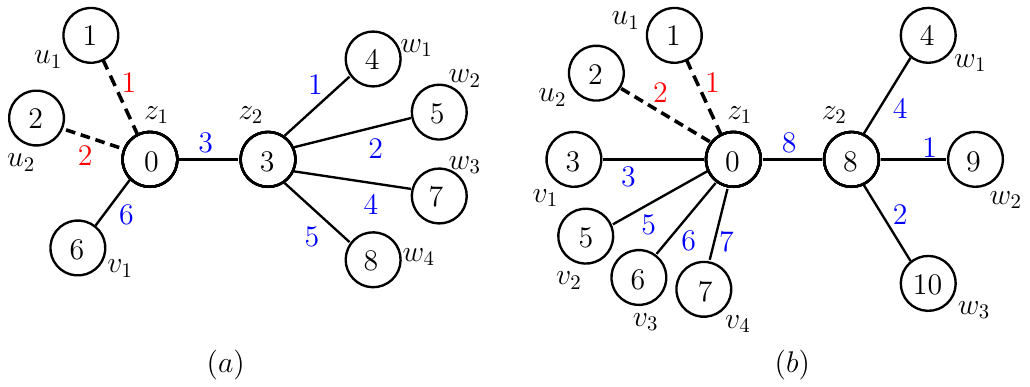}
			\caption{Additively graceful labeling of signed double stars when $n=2$ and the non-pendant edge is positive.}
			\label{fig:fig2015used}
		\end{figure}
		
		\begin{rmk}
			\label{rmk:1}
			The labeling illustrated in \Cref{fig:fig2015used}(a) shows that \Cref{thm: double star with 2 neg pendants and l=1 r>4 is not grace elegant} does not hold if $r=4$. 
		\end{rmk}
		
		\begin{thm}
			\label{thm: double star with 2 neg pendants with l=even and r l>=2 is grace elegant}
			Let $S$ be the $(p,m,2)$ signed graph consisting of a positive edge, to which are attached $l$ positive pendant edges and 2 negative pendant edges at one end as well as $r$ positive pendant edges at the other end.
			If $l,r\geq2$ and $l$ is even then, $S$ is additively graceful. Further, if $r=2$ or $3$ then, for each $l\geq 2$, $l \equiv0 \pmod2$, the additively graceful labeling obtained is unique up to rearrangement of labels among the $u_i$'s, $v_i$'s and $w_i$'s. 
		\end{thm}
		
		\begin{proof}
			Let $z_1z_2$ be the positive edge, to which are attached positive edges $z_1v_i$, $i=1,2, \dots, l$ and $z_2w_i$, $i=1,2, \dots, r$, as well as negative edges $z_1u_1$ and $z_1u_2$.
			Since \Cref{thm: double star with 2 neg pendants and l=2<=r is grace elegant} proves the result for $l=2$, hence, we may assume that $l\geq4$. Also since $l$ is even hence, $l=2t$ for some $t \in \mathbb{N}$. 
			\\\\		
			Define $f:V(S)\rightarrow\{0,1,\dots,2t+r+3\}$ as follows. $f(z_1,u_1,u_2)=(0,1,2)$
			\begin{equation*}
				\label{eqn: double star with 2 neg pendants with l=even and r,l>=2 is grace elegant}
				f(x)= \begin{cases}
					
					2+i & \text{for } x=v_i~~~i=1,2,\dots,t-1 \\
					r+i & \text{for } x=v_i~~~i=t,t+1,\dots,2t\\	
					2t+r+1 & \text{for } x=z_2 \\
					t+1+i & \text{for } x=w_i~~~i=1,2,\dots,r-2\\
					2t+3+i & \text{for } x=w_i~~~i=r-1,~r
				\end{cases}
			\end{equation*}
			\Everify, $f(V(S))=\{0,1,\dots,2t+r+3\}$ and $f$ is an additively graceful labeling of $S$.
			\\\\
			Further, suppose $l\geq 2$, $l \equiv0 \pmod2$ and $r=2$ or $3$.
			\\\\		
			If $r=2$. By \Cref{thm: double star with 2 neg pendants and r=2 is add grace}, the additively graceful labeling obtained is unique up to rearrangement of labels among the $u_i$'s, $v_i$'s and $w_i$'s.
			\\\\
			If $r=3$. Let $g$ be any additively graceful labeling of $S$. \Cref{obs: two negative edges must be adjacent in add grace sigraph} implies that, $g(z_1)=0$ and $g$ assigns labels $1$ and $2$ to vertices $u_1$ and $u_2$. \Wlg, we may assume that, $g(w_2)=l+5$ and $g(w_3)=l+6$, since $m=l+4$ and hence these labels cannot be assigned to vertices adjacent to $z_1$.
			\\\\
			Now if $f(z_2)<l+3$ then neither edge label 1 nor 2 can occur on edge $z_2w_2$ or $z_2w_3$. It follows that atleast one among edge labels 1 or 2 must occur on a positive edge adjacent to $z_1$, which is impossible. Further, if $g(z_2)=l+3$ then, $g(w_1)=l+2 \text{ or } l+4$, as this is the only way to obtain edge label 1 on a positive edge. However, 
			since $l\geq 2$ hence, $l+2$ and $l+4$, both of which need to occur as edge labels on a positive edge, are not equal to $g^*(z_2w_2)$ or $g^*(z_2w_3)$. It follows that one among $l+2$ or $l+4$ cannot be obtained as edge labels on a positive edge. Therefore,  $g(z_2)=l+4$.
			\\\\
			The remaining labels $3,4,\dots,l+3$ need to be assigned to vertices $w_1$, $v_1,v_2\dots,v_l$, so as to again obtain edge labels $3,4,\dots,l+3$. The $l$ labels assigned to each $v_i$ will generate themselves as edge labels on edge $z_1v_i$. Therefore, $g(w_1)$ needs to satisfy
			\[
			g(w_1)=g^*(z_2w_1)
			\]
			This equation has a unique solution $g(w_1)=\frac{l}{2}+2$. The remaining $l$ labels get assigned to $v_i$'s.
			Therefore, $g=f$, up to rearrangement of labels among the $u_i$'s, $v_i$'s and $w_i$'s.
		\end{proof}
				
		\begin{thm}
			\label{thm: add gracefulness of double star with one neg edge in middle and one neg pendant}
			Let $S$ be the $(p,m,2)$ signed graph consisting of a negative edge, to which are attached $l$ positive pendant edges and a negative pendant edge at one end and $r$ positive pendant edges at the other end. Then, $S$ is additively graceful if and only if $l=0$.
		\end{thm}
		
		\begin{proof}
			Let $z_1z_2$ be the negative edge, to which are attached positive edges $z_1v_i$, $i=1,2, \dots, l$ and $z_2w_i$, $i=1,2, \dots, r$, as well as a negative edge $z_1u$.
			\\\\
			First, suppose $l\geq1$.
			
			If $r=0$ then, $S$	is a star with two negative edges, which by \Cref{thm: stars are additively graceful iff n=1}, is not additively graceful.
			
			If $r\geq1$ then, let $f:V(S) \rightarrow \{0,1,\dots,l+r+2\}$ be an additively graceful labeling of $S$. By \Cref{obs: two negative edges must be adjacent in add grace sigraph}, $f(z_1)=0$ while vertices $u$ and $z_2$ must be assigned labels 1 and 2. If $f(z_2)=1$ then it is impossible to obtain the edge label 1 on a positive edge. Hence, we may assume that $f(u)=1$ and $f(z_2)=2$.	
			Now, the only way to then obtain the edge label 1 on a positive edge is to  set $f(w_i)=3$ for some $i$. Subsequently, the only way to obtain the edge label 2 on a positive edge is to set $f(w_j)=4$ for some $j$, and so on. Hence, \wlg, we may set $f(w_i)=i+2$ for $i=1,2,\dots,r$. It is now impossible to obtain the required edge label $r+1$ on a positive edge.
			
			Conversely, if $l=0$ then, define $f:V(S)\rightarrow\{0,1,\dots,r+2\}$ by $f(u,z_1,z_2) = (1,0,2)$ and $f(w_i)=2+i$, $i=1,2,\dots,r$. \Everify, $f(V(S))=\{0,1,\dots,r+2\}$ and $f$ is an additively graceful labeling of $S$.
		\end{proof}

		\begin{thm}
			\label{thm: double star with 2 neg pendants and r=odd r>=3 l=3 is not add graceful}
			Let $S$ be the $(p,m,2)$ signed graph consisting of a positive edge, to which are attached $3$ positive pendant edges and 2 negative pendant edges at one end as well as $r$ positive pendant edges at the other end.
			If $r$ is odd and $r\geq7$ then, $S$ is not additively graceful.
		\end{thm}

		\begin{proof}
			Let $z_1z_2$ be the positive edge, to which are attached positive edges $z_1v_i$, $i=1,2,3$ and $z_2w_i$, $i=1,2,\dots,r$, as well as negative edges $z_1u_1$ and $z_1u_2$. Suppose $f:V(S)\rightarrow\{0,1,\dots,r+6\}$ is an additively graceful labeling of $S$. By \Cref{obs: two negative edges must be adjacent in add grace sigraph} and \wlg, $f(z_1,u_1,u_2)=(0,1,2)$. Now $f$ needs to assign the labels $3,4,\dots,r+6$ to the vertices $v_1,v_2,v_3,z_2,w_1,w_2\dots, w_r$. Let $T=f(\{v_1,v_2,v_3\})$ and $W=f(\{w_1,w_2,\dots, w_r\})$. Since $m= r+4$ hence, $r+5,~r+6 \in W$.
			\\\\
			If $f(z_2)=r+4$ then $r+2,~r+3 \in T$ because this is the only way to obtain them as edge labels on positive edges. We now need to assign the $r-1$ labels $3,4,\dots,r+1$ to the remaining vertices in $W$ and $T$, so as to again obtain edge labels $3,4,\dots,r+1$, on the remaining positive edges. These $r-1$ labels can be arranged in pairs $(3,r+1)$, $(4,r)$, $\dots$, $(\frac{r+3}{2},\frac{r+5}{2})$ where the sum of each pair is $r+4$. Notice that if one number of the pair is in $W$ then the other must be in $W$ as well, else it is impossible to have the first appear as an edge label on a positive edge. Now since $r-2$ is odd hence, assigning labels to the remaining vertices  without violating this condition is not possible. Therefore $f(z_2)\neq r+4$, which in turn implies that $r+4 \in T$, since this is the only way to obtain $r+4$ as an edge label on a positive edge.
			\\\\
			Since $r\geq7$ hence $r+6\geq13$ and therefore, labels $0,1,\dots,13$ must be assigned to vertices.
			\\\\
			If $f(z_2)=3$ then, $6$ and subsequently $9$, followed by $12$ are forced to be in $T$, which is impossible. Therefore $f(z_2)\neq 3$, which in turn implies that $r+3 \in T$.
			\\\\
			If $f(z_2)=4$ then, $8$ and subsequently $12$ must be in $T$, which is impossible. Therefore $f(z_2)\neq 4$, which in turn implies that $r+2 \in T$.
			\\\\
			Hence, $f(z_2)=5$ because it is the only way to obtain $r+1$ as a label on a positive edge.
			It follows that $r-1$ cannot be obtained as a positive edge label. Hence $f$ cannot be an additively graceful labeling.		
		\end{proof}	
	\section{Conclusion} 
	At the outset of this investigation, it was somewhat surprising that an object as structurally simple as a double star could exhibit such intricate patterns. We have systematically studied and presented many existence, uniqueness and non-existence results on additively graceful signed stars and double stars in this paper.

	\bibliographystyle{plain}
	\bibliography{references}
	\end{document}